\documentclass[12pt]{amsart}
\usepackage{amssymb}
\usepackage{amsmath}
\theoremstyle{plain}
 \newtheorem{thm}{Theorem}[section]
 \newtheorem{cor}[thm]{Corollary}
 \newtheorem{lem}[thm]{Lemma}

 \theoremstyle{definition}
 \newtheorem{defn}[thm]{Definition}
 \theoremstyle{remark}
 \newtheorem{rem}[thm]{Remark}
 \numberwithin{equation}{section}
\begin{document}

\title[]
{Yamabe flow and ADM Mass on asymptotically flat manifolds}

\author{Liang Cheng, Anqiang Zhu}

\address{Liang Cheng, School of Mathematics and Statistics, Huazhong Normal University,
Wuhan, 430079, P.R. CHINA}

\email{math.chengliang@gmail.com}

\address{Anqiang Zhu, School of Mathematics and Statistics, Wuhan University,
Wuhan, 430072, P.R. CHINA}

\email{anqiangzhu@yahoo.com.cn}
\date{}

\begin{abstract}
In this paper, we investigate the behavior of ADM mass and Einstein-Hilbert functional under the Yamabe flow. Through studying the Yamabe flow by weighted spaces for parabolic operators, we show that the asymptotically flat property is preserved under the Yamabe flow.
We also obtain that ADM mass is invariant under the Yamabe flow and Yamabe flow is the gradient flow of Einstein-Hilbert functional on $n$-dimensional, $n\geq 3$, asymptotically flat manifolds with order $\tau>\frac{n-2}{2}$ for $n=3,4$ or $\tau>n-3$ for $n>4$. Moreover, we show that
ADM mass and Einstein-Hilbert functional are non-increasing under the Yamabe flow on $n$-dimensional asymptotically flat manifolds if we only assume the order $\tau>\frac{n-2}{2}$ for $n>4$.

{ \textbf{Mathematics Subject Classification} (2000): 
53C21,51P05}

{ \textbf{Keywords}:Yamabe flow; ADM mass; Einstein-Hilbert functional; Weighted spaces; Asymptotically flat manifolds  }
\end{abstract}

\thanks{}
 \maketitle

\section{Introduction}

In 1960s, R.Arnowitt, S.Deser, C.Misner (\cite{ADM1},\cite{ADM2},\cite{ADM3}) made a detailed study of isolated gravitational systems and defined the total mass (the ADM mass) of the gravitational system. They also proposed the famous positive mass conjecture for the space-like asymptotically
flat hypersurface of 4-Lorentz manifold which was first solved by R.Schoen and S.T.Yau (\cite{SY}, \cite{SY2}) using minimal surface and shortly thereafter by E.Witten \cite{W} using spinors.
The ADM mass of $n$-dimensional Riemannian manifolds is defined as
\begin{align}\label{ADM_Mass}
m(g)=\lim\limits_{r\to\infty}\frac{1}{4\omega}\int_{S_r}(\partial_j g_{ij}-\partial_ig_{jj})dS^i,
\end{align}
where $\omega$ denotes the volume of unit $n-1$-sphere in $\mathbb{R}^n$, $S_r$ denotes the Euclidean sphere with radius $r$ and $dS^i$ is the normal surface area element to $S_r$ with respect to Euclidean metric.
The Riemannian version of positive mass conjecture states that if $(M^n,g)$ is an $n$-dimensional asymptotically flat manifold $M^n$ of order $\tau>\frac{n-2}{2}$ and the scalar curvature is nonnegative and integrable, then $m(g)\geq 0$ with equality holds if and only if $M^n$ is isometric to $\mathbb{R}^n$. R.Schoen \cite{S} gave a proof of his work with S.T.Yau that the Riemannian version of positive mass conjecture for the dimension$\leq 7$ by using minimizing hypersurface. In \cite{B}, R.Bartnik  proved Riemannian version of positive mass conjecture
for $n$-dimensional spin manifolds following E.Witten's methods in \cite{W}. For the recent progress of this topic, one may see  \cite{Ei}, \cite{EHLS}, \cite{Lo}.

The Yamabe flow is defined by the
evolution equation
\begin{equation}\label{yamabe_flow_curvature}
\left\{
\begin{array}{ll}
         \frac{\partial g}{\partial t}=-Rg \quad &\text{in}\ M^n\times[0,T),\\
                  g(\cdot,0)=g_0 &\text{in}\ M^n,
\end{array}
\right.
\end{equation}
on an $n$-dimensional complete Riemannian manifold $(M^n,g_0)$, $n\geq 3$, where $g(t)$ is a family
of Riemannian metrics on the manifold $M^n$ and $R$ is the
scalar curvature of the metric
$$
g:=g(t)=u^{\frac{4}{n-2}}g_0,
$$
where $n\geq 3$ and
$u:M^n\to \mathbb{R}^+$ is a positive smooth function on $M^n$.
In the sequence of changing time by a constant scale,
(\ref{yamabe_flow_curvature}) can be written in the equivalent form
\begin{equation}\label{yamabe_flow_u}
\left\{
\begin{array}{ll}
         \frac{\partial u^N}{\partial t}=L_{g_0}u, \quad &\text{in}\ M^n\times[0,T),\\
                  u(\cdot,0)=1, &\text{in}\ M^n,
\end{array}
\right.
\end{equation}
where $N=\frac{n+2}{n-2}$, $L_{g_0}u=\Delta_{g_0}u-aR_{g_0}u$ and
$a=\frac{n-2}{4(n-1)}$.

The Yamabe flow was proposed by
R.Hamilton \cite{H89} in the 1980's as a tool for constructing
metrics of constant scalar curvature in a given conformal class. Consider the normalized Yamabe flow
\begin{equation}\label{rescaled_Yamabe_flow_1}
\left\{
\begin{array}{ll}
         \frac{\partial g}{\partial t}=(s-R)g \quad &\text{in}\ M^n\times[0,T),\\
                  g(\cdot,0)=g_0 &\text{in}\ M^n,
\end{array}
\right.
\end{equation}
on closed manifolds, where $s$ is the mean value of $r$, i.e. $s=\frac{\int_{M}R dvol}{vol(M)}$. R.Hamilton proved that the normalized Yamabe flow (\ref{rescaled_Yamabe_flow_1}) has a global solution for every initial metric. If the solution $g(t)$ of the Yamabe flow (\ref{rescaled_Yamabe_flow_1}) converges smoothly to a metric with constant scalar curvature as $t\to\infty$. Then it would give another approach to Yamabe problem solved by R.Schoen \cite{S1984}. This problem has been studied by several authors,
including  S.Brendle \cite{BS2005}\cite{BS2007}, B.Chow \cite{chow}, R. Ye \cite{Y94}, H. Schwetlick, M. Struwe \cite{SS}.

It is well-known that the ADM mass and positive mass theorem are closed related to the Yamabe problem. In fact, the positive mass theorem plays the key role in R.Schoen's proof of Yamabe problem. Notice that Yamabe flow is the parabolic analogue of the nonlinear elliptic equation in Yamabe problem. So a natural question is that what is the relationship between ADM mass and Yamabe flow? Also notice
that the Yamabe flow (\ref{yamabe_flow_curvature}) is the gradient flow of Einstein-Hilbert functional $\int_{M}Rd\text{vol}$, which is an important quantity in general relativity,
on closed manifolds in a fixed conformal class.  Therefore, it is natural to ask what is the behavior of Einstein-Hilbert functional under the Yamabe flow (\ref{yamabe_flow_curvature}) on asymptotically flat manifolds?

We remark that geometric flows are one of the powerful tools to study the masses in general relativity and related topics. For example, Huisken and Ilmanen \cite{HI} proved the Riemannian-Penrose conjecture in single back hole case by using inverse mean curvature flow, in which Hawking mass is monotone increasing.
 And H.L.Bray \cite{Br} proved the Riemannian-Penrose conjecture  by using his conformal flow, in which ADM mass is monotone non-increasing.
 Hence it is interested for the geometrists and physicists to find some intrinsic geometric flows such that the masses in general relativity are monotone or invariant under these flows.
An interesting work of X.Dai and L.Ma \cite{DM} showed that
ADM mass is invariant and well-defined on a 3-dimensional asymptotically flat manifold of order $\tau>\frac{1}{2}$ with nonnegative and integrable scalar curvature. They also obtained the similar results hold for higher dimensional asymptotically manifolds with more assumptions. For other works for studying the mass under the Ricci flow, one may see \cite{BW}, \cite{Li} and \cite{OW}.

Before presenting the main theorems of this paper,  we need the following two definitions.
\begin{defn}\label{AE_def}
A Riemannian manifold $M^n$, $n\geq 3$, with $C^{\infty}$ metric $g$ is called asymptotically flat of order $\tau$
if there exists a decomposition $M^n= M_0\cup M_{\infty}$
(for simplicity we deal only with the case of one end and the case of multiple
ends can be dealt with similarly) with $M_0$ compact and a diffeomorphism $M_{\infty}\cong \mathbb{R}^n-B(o,R_0)$ for some constant $R_0 > 0$ satisfying
\begin{align}\label{AE}
g_{ij} -\delta_{ij}\in C^{2+\alpha}_{-\tau}(M)
\end{align}
(defined in Definition \ref{elliptic_wss}) in the coordinates $\{x^i\}$ induced on $M_{\infty}$. And the coordinates $\{x^i\}$ are called asymptotic coordinates.
\end{defn}

\begin{defn}
We say that $u(x,t)$ is a fine solution of Yamabe flow, $0\leq t<t_{max}$, on a complete manifold $(M^n,g_0)$ if
$0<\delta\leq|u(x,t)|\leq C$ for $0\leq t\leq T$,
$\sup\limits_{M^n\times [0,T]}|\nabla_{g_0} u(x,t)|\leq C$ and $\sup\limits_{M^n\times [0,T]}|Rm(g)|(x,t)\leq C$ for any $T<t_{max}$, such that either $\lim\limits_{t\to t_{max}}\sup\limits_{M}|Rm|(\cdot,t)=\infty$ for $t_{max}<\infty$ or $t_{max}=\infty$, where $Rm(g)$ is the Riemannian
curvature of the metric
$g:=g(t)=u^{4/(n-2)}g_0$.
\end{defn}
The short time existence of smooth solution to Yamabe flow on noncompact manifold with bounded scalar curvature was obtained by B.L.Chen, X.P.Zhu \cite{CZ} and Y.An, L.Ma \cite{AM}. Based on their work, we shall show the fine solution to Yamabe flow always exists on asymptotically flat manifolds (see Corollary \ref{local_existence}).

The following theorem, which is the first result of the paper, shows that the Yamabe flow preserves the asymptotically flat condition (\ref{AE}).

\begin{thm}\label{A}
Let $u(x,t)$, $0\leq t<t_{max}$, be the fine solution to the Yamabe flow (\ref{yamabe_flow_u}) on an $n$-dimensional asymptotically flat manifold $(M^n,g_0)$ of any order $\tau>\frac{n-2}{2}$. Set $g(t)=u^{\frac{4}{n-2}}g_0$ and $v=1-u$. Then $v(x,t)\in C^{2+\alpha}_{-\tau}(M)$ and $g_{ij}(x,t)-\delta_{ij}\in C^{2+\alpha}_{-\tau}(M)$ for $t\in [0,t_{max})$.
\end{thm}

The second result of the paper is the following

\begin{thm}\label{main_5}
Let $u(x,t)$, $0\leq t< t_{max}$, be the fine solution to the Yamabe flow (\ref{yamabe_flow_u}) on an $n$-dimensional asymptotically flat manifold $(M^n,g_0)$ of order $\tau>\frac{n-2}{2}$. Assume that $R_{g_0}\geq 0$ and $R_{g_0}\in L^{1}(M)$, where $R_{g_0}$ is the scalar curvature of $g_0$.
Set $g(t)=u^\frac{4}{n-2}g_0$.

(i) In case of the dimension $n=3$ or $4$, ADM mass $m(g(t))$ is well-defined under the Yamabe flow (\ref{yamabe_flow_u}) for $0\leq t<t_{max}$(i.e. ADM mass is independent of the choices of the coordinates). Moreover, $m(g(t))\equiv m(g_0)$ and the Yamabe flow is the gradient flow of the Einstein-Hilbert functional $\int_{M}R dvol_{g(t)}$.

(ii)  In case of higher dimensions ($n>4$), the conclusions of (i) still hold if we assume the order $\tau>n-3$.

\end{thm}

In the conclusion (ii) of Theorem \ref{main_5}, we have the extra assumption of the order $\tau>n-3$ for higher dimensions ($n>4$). However, the order $\tau>\frac{n-2}{2}$ is enough for the well-definition of ADM mass (see Theorem \ref{Bartnik}) or the holding of positive mass theorem (conjecture). In the following theorem, we study the behavior of ADM mass and Einstein-Hilbert functional under the Yamabe flow only assuming the  order $\tau>\frac{n-2}{2}$.

\begin{thm}\label{main_1}
Let $u(x,t)$, $0\leq t< t_{max}$, be the fine solution to the Yamabe flow (\ref{yamabe_flow_u}) on an $n$-dimensional asymptotically flat manifold $(M^n,g_0)$ of order $\tau>\frac{n-2}{2}$. Assume that $R_{g_0}\geq 0$ and $R_{g_0}\in L^1(M)$. Set $g(t)=u^{\frac{4}{n-2}}g_0$. Then ADM mass $m(g(t))$ is well-defined and monotone non-increasing under the Yamabe flow (\ref{yamabe_flow_u}) for $0\leq t<t_{max}$.  Moreover, Einstein-Hilbert functional is monotone non-increasing under the Yamabe flow.
More precisely, we have
\begin{align}
\frac{d }{d t}\int_{M}Rdvol_{g(t)}\leq (1-\frac{n}{2})\int_{M}R^2dvol_{g(t)}\leq 0,
\end{align}
for $t\in[0,t_{max})$.
\end{thm}

The organization of the paper is as follows.
In section 2 we first recall the definitions of weighted spaces for elliptic operators on asymptotically flat manifolds. Then we give the evolution equation for Einstein-Hilbert function for the Yamabe flow. Next we show the fine solution to Yamabe flow exists on asymptotically flat manifolds (see Corollary \ref{local_existence}). Finally, we present a slight generalize version of the maximum principle on noncompact manifolds obtained by K.Ecker and G.Huisken \cite{EH}. The proof of this maximum principle is given in the appendix.

In section 3, by assuming Theorem \ref{A} holds, we give the proofs Theorem \ref{main_5} and Theorem \ref{main_1}. First, we get a Bernstein-Bando-Shi type gradient estimate for Yamabe flow. Then we give the proofs of Theorem \ref{main_5} and Theorem \ref{main_1}.

The section \ref{section_5} and section \ref{section_Y_w} are devoted to proving Theorem \ref{A}.
   The main tool for the proof of Theorem \ref{A} is the weighted spaces for the parabolic operators on asymptotically flat manifolds defined in section \ref{section_5}. In the section 1 of \cite{B}, R.Bartnik obtained the global estimates of the weighted spaces for elliptic operators (see Definition \ref{elliptic_wss}) close to Laplacian by using scaling arguments. Inspired by Bartnik's work, we first introduce the definitions of weighted spaces for parabolic operators on asymptotically flat manifolds in section \ref{section_5}. Next we get some basic inequalities about these weighted spaces. In section \ref{section_Y_w} we first prove the global estimates of the weighted spaces defined in section \ref{section_5} for Yamabe flow. Finally, we give the proof of Theorem \ref{A}.

\section{preliminaries}

First, we recall
the definitions of weighted spaces (see \cite{B} and \cite{LP}) for elliptic operators on asymptotically flat manifolds.

\begin{defn}\label{elliptic_wss}
Suppose $(M^n,g)$ is an $n$-dimensional asymptotically flat manifold with asymptotic coordinates $\{x^i\}$. Denote $D^j_x v=\sup\limits_{|\alpha|=j}|\frac{\partial^{|\alpha|}}{\partial x_{i_1}\cdots\partial x_{i_j}}v|$.
Let $r(x)=|x|$ on $M_{\infty}$ (defined in Definition \ref{AE_def}) and extend $r$ to a smooth positive function on all of $M^n$. For $q\geq 1$ and $\beta\in \mathbb{R}$,
the weighted Lebesgue space $L^q_{\beta}(M)$ is defined as the set of locally integrable functions $v$ for which the norm
$$ ||v||_{L^q_\beta(M)}=\left\{
                      \begin{array}{ll}
                        (\int_{M}|v|^q r^{-\beta q-n}dx)^{\frac{1}{q}}, & \hbox{$q<\infty$;} \\
                        ess \sup\limits_{M} (r^{-\beta}|v|), & \hbox{$q=\infty$,}
                      \end{array}
                    \right.
$$
is finite.
Then the weighted Sobolev space $W^{k,q}_\beta(M)$ is defined as the set of functions $v$ for which $|D^j_xv|\in L^q_{\beta-j}(M)$ with the norm
$$
||v||_{W^{k,q}_\beta(M)}=\sum\limits^k_{j=0}||D^j_x v||_{L^q_{\beta-j}(M)}.
$$
For a nonnegative integer $k$, the weighted $C^k$ space $C^k_{\beta}(M)$ is defined as the set of $C^k$ functions $v$ for which the norm
$$
||v||_{C^k_\beta(M)}=\sum\limits_{j=0}^k\sup\limits_{M} r^{-\beta+j}|D^j_xv|
$$
is finite.
Then the weighted H\"{o}lder space $C^{k+\alpha}_{\beta}(M)$ is defined as the set of functions $v\in C^{k}_{\beta}(M)$ for which the norm
$$
||v||_{C^{k+\alpha}_\beta(M)}=||v||_{C^k_\beta(M)}+\sup\limits_{x\neq y\in M}\min(r(x),r(y))^{-\beta+k+\alpha}\frac{|D^k_xv(x)-D^k_xv(y)|}{|x-y|^{\alpha}}.
$$
is fnite.
\end{defn}

 In \cite{B}, R.Bartnik proved the following theorem by using weighted spaces defined in Definition \ref{elliptic_wss} (also see Theorem 9.6 in \cite{LP}).
\begin{thm}(R.Bartnik)\label{Bartnik}
If $(M^n,g)$ is an $n$-dimensional asymptotic flat manifold with $g_{ij}-\delta_{ij}\in C^{1+\alpha}_{-\tau}(M)$, $R_g\in L^1(M)$ and $R_g\geq 0$,
where $\tau>\frac{n-2}{2}$, then ADM mass $m(g)$ defined in (\ref{ADM_Mass}) is well-defined(i.e. independent of the choices of coordinates).
\end{thm}

Second, we present the evolution equation for Einstein-Hilbert function for the Yamabe flow.
\begin{lem}
\label{L1_decay}
Let $u(x,t)$, $0\leq t< t_{max}$, be the fine solution to the Yamabe flow (\ref{yamabe_flow_u}) on an $n$-dimensional asymptotically flat manifold $(M^n,g_0)$ of order $\tau>\frac{n-2}{2}$.
Set $g(t)=u^\frac{4}{n-2}g_0$. Assume that $\int_{M}R^2dvol_{g(t)}$ is finite for $0\leq t<t_{max}$.  We have
\begin{align}\label{L1_decay_1}
\frac{d }{d t}\int_{M}Rdvol_{g(t)}=4\omega\frac{d }{d t}m(g(t))
+(1-\frac{n}{2})\int_{M}R^2dvol_{g(t)},
\end{align}
 where $m(g)$ is the ADM mass defined in (\ref{ADM_Mass}) with any asymptotic coordinates $\{x_i\}$.
\end{lem}
\begin{proof}
Let $\frac{\partial g_{ij}}{\partial t}=v_{ij}$. Then it follows from the variation of Einstein-Hilbert functional (see (8.9) and (8.11) in \cite{LP}) that
\begin{align}\label{Einstein_Variation}
&\qquad\frac{d }{d t}\int_{M}Rdvol_{g(t)}
=4\omega\frac{d }{d t}m(g(t))
-\int_{M}v^{ij}G_{ij}dvol_{g(t)},
\end{align}
  where $G_{ij}=Rc_{ij}-\frac{1}{2}Rg_{ij}$ is the Einstein tensor.
   Then Lemma \ref{L1_decay} follows from (\ref{Einstein_Variation}) and  $v_{ij}=-Rg_{ij}$ for Yamabe flow.
   For sake of convenience for the readers, we give a proof of (\ref{Einstein_Variation}) below. By the variation of scalar curvature (see Lemma 2.7 in \cite{CLN}),
$$
\frac{\partial }{\partial t}R=-\Delta (g^{ij}v_{ij})+div(div(v))-v^{ij}Rc_{ij}.
$$
Hence
$$
\frac{\partial }{\partial t}(Rdvol_{g(t)})=(-\Delta (g^{ij}v_{ij})+div(div(v))-v^{ij}(Rc_{ij}-\frac{1}{2}Rg_{ij}))dvol_{g(t)}.
$$
Then
\begin{align*}
&\qquad\frac{d }{d t}\int_{B(o,r)}Rdvol_{g(t)}\\
&=\int_{B(o,r)}(-\Delta (g^{ij}v_{ij})+div(div(v)))dvol_{g(t)}-\int_{B(o,r)}v^{ij}G_{ij}dvol_{g(t)}\\
&=\int_{S_r}< \nabla(g^{ij}v_{ij})-div(v),\nu>dS-\int_{B(o,r)}v^{ij}G_{ij}dvol_{g(t)}\\
&=\int_{S_r}<\xi,\nu>dS-\int_{B(o,r)}v^{ij}G_{ij}dvol_{g(t)},
\end{align*}
where $\xi_i=(v_{ij,j}-v_{jj,i})(1+O(r^{-1}))$, where $r=r(x)$ is the function defined in Definition \ref{elliptic_wss}.
It follows that
\begin{align}
\frac{d }{d t}\int_{M}Rdvol_{g(t)}&=\lim\limits_{r\to\infty}\int_{S_r}(v_{ij,j}-v_{jj,i})dS^i-\int_{M}v^{ij}G_{ij}dvol_{g(t)}\nonumber\\
&=4\omega\frac{d }{d t}m(g(t))
-\int_{M}v^{ij}G_{ij}dvol_{g(t)}.\nonumber
\end{align}
\end{proof}

Third, we show that the fine solution to Yamabe flow always exists on asymptotically flat manifolds.
Recall the short time existence of smooth solution to Yamabe flow (\ref{yamabe_flow_u}) on noncompact manifolds obtained by B.L.Chen and X.P.Zhu \cite{CZ} and Y.An, L.Ma \cite{AM}.
\begin{thm}\cite{CZ}\cite{AM}\label{short_existence}
If $(M^n,g_0)$ is an $n$-dimensional complete manifold with bounded scalar curvature, then Yamabe flow (\ref{yamabe_flow_u}) has
a smooth solution on a maximal time interval $[0,t_{max})$ with $t_{max}>0$ such that either $t_{max}=+\infty$ or the evolving metric contracts to a point at finite time $t_{max}$.
\end{thm}

We remark that solution $u(x,t)$ to Yamabe flow (\ref{yamabe_flow_u}) in Theorem \ref{short_existence} satisfies $0<c_1<u(x,t)<c_2$ for some constant $c_1, c_2$ on any compact subinterval  $[0,T]$, $T<t_{max}$, where $c_1$ and $c_2$ are the constants depending on $T$. In fact, the solution $u(x,t)$ to Yamabe flow (\ref{yamabe_flow_u}) in Theorem \ref{short_existence} is obtained, see \cite{CZ} and \cite{AM}, by a sequence of approximation solutions
$u_m(x,t)$ which solve the following Dirichlet problem for a sequence of exhausting bounded smooth domains
\begin{equation}\label{yamabe_flow_u_local1}
\left\{
\begin{array}{ll}
         \frac{\partial u^N_m}{\partial t}=L_{g_0}u_m, \quad &\ x\in\Omega_m, \ t>0,\\
                   u_m(x,t)>0, \quad &\ x\in\Omega_m, \ t>0,\\
                   u_m(x,t)=1, \quad &\ x\in \partial \Omega_m, \ t>0,\\
                  u_m(\cdot,0)=1, \quad &\ x\in\Omega_m,\\
\end{array}
\right.
\end{equation}
where $\Omega_1\subset\Omega_2\subset\cdots$, $N=\frac{n+2}{n-2}$, $L_{g_0}u_m=\Delta_{g_0}u_m-aR_{g_0}u_m$ and
$a=\frac{n-2}{4(n-1)}$. Since  $u_m(x,t)=1$ is bounded on $\partial \Omega_m$,  by the maximum principle, we conclude that
\begin{align*}
\max\limits_{\Omega_m} u_m(t)\leq (1+\frac{n-2}{(n-1)(n+2)}\sup\limits_{M^n}|R_{g_0}|t)^{\frac{n-2}{4}}.
\end{align*}
and
\begin{align*}
\min\limits_{\Omega_m} u_m(t)\geq (1-\frac{n-2}{(n-1)(n+2)}\sup\limits_{M^n}|R_{g_0}|t)^{\frac{n-2}{4}}.
\end{align*}
We see that $u_m(t)$ has an uniformly upper bound on $[0,t_{\max})$ for $t_{max}<\infty$ and uniformly positive lower bound on $[0,\frac{(n-1)(n+2)}{2(n-2)\sup\limits_{M^n}|R_{g_0}|}]$.
We then consider
\begin{equation}\label{yamabe_flow_u_local1}
\left\{
\begin{array}{ll}
         \frac{\partial u^N_m}{\partial t}=L_{g_0}u_m, \quad &\ x\in\Omega_m, \ t>t_0,\\
                   u_m(x,t)>0, \quad &\ x\in\Omega_m, \ t>t_0,\\
                   u_m(x,t)=1, \quad &\ x\in \partial \Omega_m, \ t>t_0,\\
                  u_m(\cdot,t_0)=u^m_0(x), \quad &\ x\in\Omega_m,\\
\end{array}
\right.
\end{equation}
where $0<\delta\leq u^m_0(x)\leq C, x\in\Omega_m$. Again by the maximum principle, we have
\begin{align*}
\min\limits_{\Omega_m} u_m(x,t)\geq (\delta^{\frac{4}{n-2}}-\frac{n-2}{(n-1)(n+2)}\sup\limits_{M^n}|R_{g_0}|(t-t_0))^{\frac{n-2}{4}}.
\end{align*}
 Then we get that there exists a constant
 \begin{equation}\label{rrrrrrrrrrrrrrrrr}
 \Delta T(\delta,\sup\limits_{M^n}|R_{g_0}|)=\frac{(n-1)(n+2)\delta^{\frac{4}{n-2}}}{2(n-2)\sup\limits_{M^n}|R_{g_0}|}
 \end{equation}
  such that
the Dirichlet problem (\ref{yamabe_flow_u_local1}) has the unique solution satisfying $u_m(x,t)\geq 2^{-\frac{n-2}{4}}\delta$ on $[t_0,t_0+\Delta T(\delta,\sup\limits_{M^n}|R_{g_0}|)]$.
 For $t_0=0$ and $u_m(\cdot,0)=1$ in (\ref{yamabe_flow_u_local1}), there exists $T_1=\Delta T_1(1,\sup\limits_{M^n}|R_{g_0}|)$ such that
$u_m(x,t)\geq 2^{-\frac{n-2}{4}}$ on $[0,T_1]$. Take $\delta_1=\inf\limits_{m,\Omega_m}u_m(x,T_1)$. For $t_0=T_{k-1}$  and $u^m_0(x)=u_m(\cdot,T_{k-1})$  in (\ref{yamabe_flow_u_local1}), there exists $T_k(\delta_{k-1},\sup\limits_{M^n}|R_{g_0}|)=T_{k-1}+\Delta T_{k-1}(\delta_{k-1},\sup\limits_{M^n}|R_{g_0}|)$ such that $u_m(x,t)\geq 2^{-\frac{n-2}{4}}\delta_{k-1}$ on $[T_{k-1},T_k]$. Take $\delta_k=\inf\limits_{m,\Omega_m}u_m(x,T_k)$.
Then we have either $T_k\to T<\infty$ or $T_k\to \infty$. If $T_k\to T<\infty$, then $\Delta T_k \to 0$.  By (\ref{rrrrrrrrrrrrrrrrr}), we see that $\delta_k\to 0$ and $T=t_{max}$ is the blow-up time such that $\inf_{M} g(\cdot,t)\to 0$ as $t\to T$.
This implies the solution $u(x,t)$ to Yamabe flow (\ref{yamabe_flow_u}) in Theorem \ref{short_existence} satisfies $0<c_1<u(x,t)<c_2$ for some constants $c_1, c_2$ on any compact subinterval  $[0,T]$ for $T<t_{max}$.

Based on Theorem \ref{short_existence} and the remarks above, we can show that there exists a fine solution to Yamabe flow (\ref{yamabe_flow_u}) on asymptotically flat manifolds.
\begin{cor}\label{local_existence}
If $(M^n,g_0)$ is an $n$-dimensional asymptotically flat manifold of order $\tau>0$, then Yamabe flow (\ref{yamabe_flow_u}) has
a fine solution $u(x,t)$ on a maximal time interval $[0,t_{max})$ with $t_{max}>0$.
\end{cor}
\begin{proof}
By Theorem \ref{short_existence}, there exists a smooth solution $u(x,t)$ to Yamabe flow (\ref{yamabe_flow_u}) on a maximal time interval $[0,t_{max})$ with $t_{max}>0$ such that $0<c_1<u(x,t)<c_2$ for some constants $c_1, c_2$ on any compact subinterval $[0,T]$, $T<t_{max}$. Let $r_0$ be a fixed positive constant. Applying the Krylov-Safonov estimate and Schauder estimate of parabolic equations to (\ref{yamabe_flow_u}) on $B_{g_0}(p,r_0)$ (see \cite{L}), we have
$||u||_{C^{2+\alpha,1+\frac{\alpha}{2}}(B_{g_0}(p,r_0)\times [0,T])}\leq C$, where $C$ is independent of the point $p$.
Since $g(\cdot,t)=u^{\frac{4}{n-2}}g_0$, we have $\sup\limits_{B_{g_0}(p,r_0)\times [0,T]}|Rm(x,t)|\leq C$.
\end{proof}

Finally, we present a slight generalized version of the maximum principle obtained by K.Ecker and G.Huisken (see Theorem 4.3 in \cite{EH}), where they consider
the maximum principle for the parabolic equation $\frac{\partial }{\partial t}v - \Delta v\leq b\cdot \nabla v+cv$ on noncompact manifolds, where $\Delta$ and $\nabla$ depend on $g(t)$. With little more observation, K.Ecker and G.Huisken's maximum principle can be easily generalized to fitting the equation $\frac{\partial }{\partial t}v -  div(a \nabla v)\leq b\cdot \nabla v+cv$.
 We shall give a proof of Theorem \ref{Ecker_Huisken} in the appendix for sake of convenience for the readers.

\begin{thm}\label{Ecker_Huisken}
Suppose that the complete noncompact manifold $M^n$ with Riemannian metric $g(t)$ satisfies the uniformly volume
growth condition
\begin{align*}
    vol_{g(t)}(B_{g(t)}(p,r))\leq exp(k(1+r^2))
\end{align*}
for some point $p\in M$ and a uniform constant $k>0$ for all $t\in [0,T]$. Let $v$ be a differentiable function on $M\times (0,T]$ and
continuous on $M\times [0,T]$. Assume that $v$ and $g(t)$ satisfy

(i) the differential inequality
\begin{align*}
    \frac{\partial }{\partial t}v - div(a \nabla v)\leq b\cdot \nabla v+cv,
\end{align*}
where the vector field $b$ and the function $a$ and $c$ are uniformly bounded
\begin{align*}
    0<\alpha_1' \leq a\leq \alpha_1, \sup\limits_{M\times [0,T]} |b|\leq \alpha_2, \sup\limits_{M\times [0,T]} |c|\leq \alpha_3,
\end{align*}
for some constants $\alpha_1',\alpha_1,\alpha_2<\infty$. Here $\Delta$ and $\nabla$ depend on $g(t)$.

(ii) the initial data
\begin{align*}
    v(p,0)\leq 0,
\end{align*}
for all $p\in M$.

(iii) the growth condition
\begin{align*}
    \int^T_0(\int_{M}exp[-\alpha_4 d_{g(t)}(p,y)^2]|\nabla v|^2(y)d \mu_t)dt<\infty.
\end{align*}
for some constant $\alpha_4>0$.

(iv) bounded variation condition in metrics
\begin{align*}
    \sup\limits_{M\times[0,T]}|\frac{\partial}{\partial t}g(t)|\leq \alpha_5
\end{align*}
for some constant $\alpha_5<\infty$.

Then, we have
\begin{align*}
    v\leq 0
\end{align*}
on $M\times [0,T]$.
\end{thm}
\begin{rem}\label{remark_max}
Clearly, the conditions (iii) and (iv) are satisfied if the sectional curvature of $g(t)$ and $\nabla v$ are uniformly bounded on $[0,T]$.
\end{rem}

\section{Proof of Theorem \ref{main_5} and Theorem \ref{main_1}}
In this section, by assuming Theorem \ref{A} holds, we give the proofs of Theorem \ref{main_5} and Theorem \ref{main_1}. The proof of Theorem \ref{A} will be presented in section \ref{section_5} and section \ref{section_Y_w}.

We first recall some basic formulas for Yamabe flow from \cite{chow} (see Lemmas 2.2 and
2.4 in \cite{chow}).

\begin{lem}\label{Chow}\cite{chow}  Let $g(t)$, $n\geq 3$, be the solution to the Yamabe flow (\ref{yamabe_flow_curvature}), then the scalar curvature of $g(t)$ evolves by
\begin{equation}\label{scalar}
\frac{\partial }{\partial t}R=(n-1)\Delta R+R^2.
 \end{equation}
 Moreover, if the initial metric $g_0$ is
locally conformally flat, then the evolution equation of the Ricci curvature is
\begin{equation}\label{ric}
\frac{\partial}{\partial t} R_{ij}=(n-1)\Delta R_{ij}+\frac{1}{n-2}B_{ij},
\end{equation}
where
$$
B_{ij}=(n-1)|Ric|^2g_{ij}+nRR_{ij}-n(n-1)R_{ij}^2-R^2g_{ij}.
$$
\end{lem}
\begin{rem}
In view of (\ref{ric}), the Ricci curvature evolves as a parabolic equation only if the initial metric  $g_0$ is
locally conformally flat. So the Yamabe flow can not smooth out the metric as the Ricci flow in general.
\end{rem}

As an immediate application to (\ref{scalar}) in Theorem \ref{Chow} and maximum principle (see Corollary 7.43 in \cite{CLN}),the following corollary holds.
\begin{cor}\label{positive}
Let $g(t)$ be the fine solution to Yamabe flow (\ref{yamabe_flow_curvature}) on an $n$-dimensional asymptotically flat manifold $(M^n,g_0)$ of order $\tau>0$. If $R_{g_0}\geq 0$, then $R(g(t))\geq 0$ for $0\leq t<t_{max}$.
\end{cor}

Before presenting the proof of Theorem \ref{main_5}, we need the following Bernstein-Bando-Shi type gradient estimate for the Yamabe flow.
\begin{thm}\label{gradient_estimates_for_R}
Let $g(t)$ be the solution to Yamabe flow (\ref{yamabe_flow_curvature}) on an $n$-dimensional manifold $(M^n,g_0)$.
Let $K<\infty$ be positive constants. For each $r>0$, there is a constant
$C(n)$ such that
if
\begin{align}\label{gradient_estimates_for_R1}
    |Rm(x,t)|\leq K \ \text{for}\ (x,t)\in B_{g_0}(p,r)\times [0,\tau],
\end{align}
then
\begin{align*}
    |\nabla R(x,t)|\leq C(n)K(\frac{1}{r^2}+\frac{1}{\tau}+K)^{\frac{1}{2}}
\end{align*}
for all $x\in B_{g_0}(p,\frac{r}{2})$ and $t\in (0,\tau]$.
\end{thm}
\begin{proof}
We just proceed as W.X.Shi's gradient estimate (\cite{S892}, \cite{S89}) of Ricci flow.
From (\ref{scalar})
we obtain
\begin{align*}
    \frac{\partial}{\partial t} R^2&=(n-1)(\Delta R^2-2|\nabla R|^2)+2R^3,
\end{align*}
and
\begin{align*}
    \frac{\partial}{\partial t}|\nabla R|^2&=(n-1)(\Delta |\nabla R|^2-2|\nabla^2 R|^2)+5R|\nabla R|^2-2(n-1)Rc(\nabla R,\nabla R)\\
    &\leq (n-1)(\Delta |\nabla R|^2-2|\nabla^2 R|^2)+C(n)K|\nabla R|^2,
\end{align*}
where $C(n)$ depends only on $n$ and $Rc$ denotes the Ricci curvature.
Then we compute
\begin{align*}
    &\quad(\frac{\partial}{\partial t}-(n-1)\Delta)((16K^2+R^2)|\nabla R|^2)\\
&\leq |\nabla R|^2(2R^3-2(n-1)|\nabla R|^2)\\
&\quad+(16K^2+R^2)(-2(n-1)|\nabla^2 R|^2+C(n)K|\nabla R|^2))\\
&\quad+8(n-1)|R||\nabla R|^2|\nabla^2R|\\
    &\leq -2(n-1)|\nabla R|^4+2K^3|\nabla R|^2-32(n-1)K^2|\nabla^2 R|^2\\
&\quad+17C(n)K^3|\nabla R|^2+8(n-1)K|\nabla R|^2|\nabla^2R|.
\end{align*}
Since
$$
-\frac{1}{2}|\nabla R|^4+8K|\nabla R|^2|\nabla^2 R|-32K^2|\nabla^2R|^2\leq 0,
$$
and
$$
-\frac{1}{2}|\nabla R|^4+(2+17C(n)K^3)|\nabla R|^2\leq C'(n)K^6,
$$
we have
\begin{align*}
    (\frac{\partial}{\partial t}-(n-1)\Delta)((16K^2+R^2)|\nabla R|^4)\leq -|\nabla R|^4+C'(n)K^6.
\end{align*}
Taking $G=\min\{\frac{1}{289},\frac{1}{C'(n)}\}\frac{(16K^2+R^2)|\nabla R|^2}{K^4}$, we get
\begin{align}\label{gradient_es}
    (\frac{\partial}{\partial t}-(n-1)\Delta)G\leq -G^2+K^2.
\end{align}
Note that (\ref{gradient_es}) shows that it just the same situation as which has been studied in W.X.Shi's gradient estimate for Ricci flow. So we omit the details here, and one may see
 Lemma 6.19 and Lemma 6.20 in \cite{CLN} for the proof.
\end{proof}
\begin{rem}
If the initial metric $g_0$ is locally conformally flat and the solution to Yamabe flow on $(M^n,g_0)$ satisfies the assumptions of Theorem \ref{gradient_estimates_for_R}, by applying the similar methods in the proof of Theorem \ref{gradient_estimates_for_R} to (\ref{ric}), one have
\begin{align*}
    |\nabla Rc(x,t)|\leq C(n)K(\frac{1}{r^2}+\frac{1}{\tau}+K)^{\frac{1}{2}}
\end{align*}
for all $x\in B_{g_0}(p,\frac{r}{2})$ and $t\in (0,\tau]$ (see \cite{MC}).
\end{rem}

Finally, we give the proofs of Theorem \ref{main_5} and Theorem \ref{main_1}.

\textbf{Proof of Theorem \ref{main_5}:}
 Denote $m(g,\{x_i\})$ be the ADM mass defined in (\ref{ADM_Mass}) with metric $g$ and asymptotic coordinates $\{x_i\}$.
 We calculate that
 \begin{align*}
 \frac{d}{dt}m(g(t),\{x_i\})&=\lim\limits_{r\to\infty}\frac{1}{4\omega}\int_{S_r}((Rg_{jj})_{,i}-(Rg_{ij})_{,j})dS^i\\
 &=\lim\limits_{r\to\infty}\frac{1}{4\omega}\int_{S_r}R(g_{jj,i}-g_{ij,j})dS^i\\
 &\qquad+ \lim\limits_{r\to\infty}\frac{1}{4\omega}\int_{S_r}(R_{,i}g_{jj}-R_{,j}g_{ij})dS^i
 \end{align*}
 By Theorem \ref{A}, $g(t)$ satisfies the asymptotic conditions (\ref{AE}) of order $\tau>\frac{n-2}{2}$ for $0\leq t< t_{max}$. Then $g_{ij}(\cdot,t)=\delta_{ij}+O(r^{-\tau})$, $g_{ij,k}(\cdot,t)=O(r^{-\tau-1})$ and $|Rm(g(t))|=O(r^{-\tau-2})$ for $\tau>\frac{n-2}{2}$, $0\leq t< t_{max}$. So $\lim\limits_{r\to\infty}\frac{1}{4\omega}\int_{S_r}R(g_{jj,i}-g_{ij,j})dS^i=0$.
 Note  that $|\nabla R|=O(r^{-\tau-2})$ for  $0\leq t< t_{max}$ by Theorem \ref{gradient_estimates_for_R}.
 It follows that
  \begin{align*}
 \frac{d}{dt}m(g(t),\{x_i\})&= \lim\limits_{r\to\infty}\frac{n-1}{4\omega}\int_{S_r}R_{,i}dS^i+ \lim\limits_{r\to\infty}\frac{1}{4\omega}\int_{S_r}R_{,i}O(r^{-\tau})dS^i\\
  &=\lim\limits_{r\to\infty}\frac{n-1}{4\omega}\int_{S_r}R_{,i}dS^i.
 \end{align*}
 Then we have
$ \frac{d}{dt}m(g(t),\{x_i\})=0$ when $\tau>\max\{\frac{n-2}{2},n-3\}$ (i.e. $\tau>\frac{n-2}{2}$ when $n=3,4$ or $\tau>n-3$ when $n>4$).
Since $R(g(t))=O(r^{-\tau-2})$ for $\tau>\frac{n-2}{2}$, $0\leq t< t_{max}$, we get $R^2(g(t))$ is integrable.
It follows from Theorem \ref{Einstein_Variation} that
\begin{align}
\frac{d }{d t}\int_{M}Rdvol_{g(t)}=(1-\frac{n}{2})\int_{M}R^2dvol_{g(t)}\leq 0.
\end{align}
Then $R(g(t)) \in L^1(M)$ and the Yamabe flow (\ref{yamabe_flow_curvature}) is the gradient flow of Einstein-Hilbert functional.
Hence the ADM mass is well-defined on $[0,t_{max})$ by Theorem \ref{A}, Theorem \ref{Bartnik} and Corollary \ref{positive}.
 $\Box$

\textbf{Proof of Theorem \ref{main_1}:}
Recall the solution $u(x,t)$ to Yamabe flow (\ref{yamabe_flow_u}) in Theorem \ref{short_existence} is obtained, see \cite{CZ} and \cite{AM}, by a sequence of approximation solutions
$u_m(x,t)$ which solve the following Dirichlet problem for a sequence of exhausting bounded smooth domains
where $\Omega_1\subset\Omega_2\subset\cdots$:
\begin{equation}\label{yamabe_flow_u_local1}
\left\{
\begin{array}{ll}
         \frac{\partial u^N_m}{\partial t}=L_{g_0}u_m, \quad &\ x\in\Omega_m, \ t>0,\\
                   u_m(x,t)>0, \quad &\ x\in\Omega_m, \ t>0,\\
                   u_m(x,t)=1, \quad &\ x\in \partial \Omega_m, \ t>0,\\
                  u_m(\cdot,0)=1, \quad &\ x\in\Omega_m,\\
\end{array}
\right.
\end{equation}
where $\Omega_1\subset\Omega_2\subset\cdots$, $N=\frac{n+2}{n-2}$, $L_{g_0}u_m=\Delta_{g_0}u_m-aR_{g_0}u_m$ and
$a=\frac{n-2}{4(n-1)}$.
Without lose of generality, we may assume $\Omega_m=B(o,m)$, where $B(o,m)$ is the Euclidean ball with radius $m$ centered at $o$. Fix time $t_0>0$ and set $g_m(t)=u_m^{\frac{4}{n-2}}g_0$. We also consider
the Dirichlet problem that
\begin{equation}
\left\{
\begin{array}{ll}
         \frac{\partial g_m}{\partial t}=-Rg_m, \quad &\ x\in\Omega_m, \ t>t_0,\\
                  g_m(t)=g_m(t_0), \quad &\ x\in\partial\Omega_m, t>t_0,
\end{array}
\right.
\end{equation}
where $g_m(x,t)=\widetilde{u}_m^{\frac{4}{n-2}}(x,t)g_m(x,t_0)$. By the uniqueness of the Dirichlet problem,
we have $\widetilde{u}_m(x,t)=\frac{u_{m}(x,t)}{u_{m}(x,t_0)}$. Note that $\widetilde{u}_m$ satisfies the following equation
\begin{equation}
\left\{
\begin{array}{ll}
         \frac{\partial \widetilde{u}^N_m}{\partial t}=L_{g_m(t_0)}\widetilde{u}_m, \quad &\ x\in\Omega_m, \ t>t_0,\\
                   \widetilde{u}_m(x,t)>0, \quad &\ x\in\Omega_m, \ t>t_0,\\
                  \widetilde{ u}_m(x,t)=1, \quad &\ x\in \partial \Omega_m, \ t>t_0,\\
                  \widetilde{u}_m(\cdot,t_0)=1, \quad &\ x\in\Omega_m,\\
\end{array}
\right.
\end{equation}
By the maximal principle and $R_{g_0}\geq 0$ that $\widetilde{u}_m(x,t)\leq 1$ on $\Omega_m$ for $t\geq t_0$. Since $\widetilde{u}_m=1$ on $\partial \Omega_m$, we deduce that
$\frac{\partial \widetilde{u}_m}{\partial \nu}\geq 0$ on $\partial \Omega_m$, where $\nu$ is the outer unit normal vector with respect to Euclidean metric.
We denote $(\xi_m(t))_i=(g_m(t))_{ij,j}-(g_m(t))_{jj,i}$, $<,>$ denotes the Euclidean inner product, $S_m=\partial B(o,m)$ and $dS_m$ is the volume element of $S_m$ with respect to the Euclidean metric.
We calculate
\begin{align}\label{mass_f_111}
   &\qquad\frac{1}{4\omega}\int_{S_m}<\xi_m(t),\nu>dS_m\nonumber\\
&=\frac{1}{4\omega}\int_{S_m}\widetilde{u}_m(x,t)^{\frac{4}{n-2}}<\xi_m(t_0),\nu>dS_m\nonumber\\
        &\quad+\frac{1}{(n-2)\omega}\int_{S_m}\widetilde{u}_m(x,t)^{\frac{6-n}{n-2}}(\widetilde{u}_m(x,t)_{,j}(g_m)_{ij}(t_0)-\widetilde{u}_m(x,t)_{,i}g(t_0)_{jj})\nu^idS_m\nonumber\\
        &=\frac{1}{4\omega}\int_{S_m}\widetilde{u}_m(x,t)^{\frac{4}{n-2}}<\xi_m(t_0),\nu>dS_m\nonumber\\
        &\quad+\frac{1}{(n-2)\omega}\int_{S_m}\widetilde{u}_m(x,t)^{\frac{6-n}{n-2}}((n-1)\widetilde{u}_m(x,t)_{,i}+\widetilde{u}_m(x,t)_{,j}(g_m)_{ij}(t_0)\nonumber\\
        &\quad-\widetilde{u}_m(x,t)_{,i}g(t_0)_{jj})\nu^idS_m-\frac{1}{(n-2)\omega}\int_{S_m}(n-1)\widetilde{u}_m(x,t)^{\frac{6-n}{n-2}}\frac{\partial \widetilde{u}_m(x,t)}{\partial \nu}dS_m\nonumber\\
        &\leq\frac{1}{4\omega}\int_{S_m}\widetilde{u}_m(x,t)^{\frac{4}{n-2}}<\xi_m(t_0),\nu>dS_m\nonumber\\
        &\quad+\frac{1}{(n-2)\omega}\int_{S_m}\widetilde{u}_m(x,t)^{\frac{6-n}{n-2}}((n-1)\widetilde{u}_m(x,t)_{,i}+\widetilde{u}_m(x,t)_{,j}(g_m)_{ij}(t_0)\nonumber\\
        &\quad-\widetilde{u}_m(x,t)_{,i}g(t_0)_{jj})\nu^idS_m
         \end{align}
Note that $0<\delta\leq|u(x,t)|\leq C$ and $\sup\limits_{M^n\times [0,T]}|\nabla_{g_0} u(x,t)|\leq C$ on time interval $[0,T]$.
By Theorem \ref{A}, we know that $u(x,t)\to 1$ as $r\to\infty$ and $u(x,t)_{,i}=O(r^{-(\tau+1)})$ for $\tau>\frac{n-2}{2}$. Since $\widetilde{u}_m(x,t)=\frac{u_{m}(x,t)}{u_{m}(x,t_0)}$, we have $\widetilde{u}(x,t)=\frac{u(x,t)}{u(x,t_0)}$ by letting $m\to\infty$. Then we have that $\widetilde{u}(x,t)\to 1$ as $r\to\infty$ and $\widetilde{u}(x,t)_{,i}=O(r^{-(\tau+1)})$ for $\tau>\frac{n-2}{2}$.
Then taking $m\to\infty$ in (\ref{mass_f_111}), we conclude that
\begin{align*}
       m(g(t),\{x_i\}) &\leq m(g(t_0),\{x_i\})+\frac{1}{(n-2)\omega}\lim\limits_{m\to\infty}\int_{S_m}\widetilde{u}(x,t)^{\frac{6-n}{n-2}}((n-1)\widetilde{u}(x,t)_{,i}\nonumber\\
        &\quad+\widetilde{u}(x,t)_{,j}(g(t_0))_{ij}-\widetilde{u}(x,t)_{,i}g(t_0)_{jj})\nu^idS_m\\
        &=m(g(t_0),\{x_i\})+\frac{1}{(n-2)\omega}\lim\limits_{m\to\infty}\int_{S_m}\widetilde{u}(x,t)^{\frac{6-n}{n-2}}((n-1)\widetilde{u}(x,t)_{,i}\nonumber\\
        &\quad+\widetilde{u}(x,t)_{,j}(\delta_{ij}+O(r^{-\tau}))-\widetilde{u}(x,t)_{,i}(n+O(r^{-\tau}))\nu^idS_m\\
        &=m(g(t_0),\{x_i\})+\frac{1}{(n-2)\omega}\lim\limits_{m\to\infty}\int_{S_m}\widetilde{u}(x,t)^{\frac{6-n}{n-2}}\widetilde{u}(x,t)_{,i}\nu^iO(r^{-\tau})dS_m\\
        &=m(g(t_0),\{x_i\}).
\end{align*}
Since $R(g(t))=O(r^{-(2+\tau)})$ for $\tau>\frac{n-2}{2}$ by Theorem \ref{A}, we get $R^2(g(t))$ is integrable for $0\leq t<t_{max}$.
It follows from Theorem \ref{Einstein_Variation} that
\begin{align}
\frac{d }{d t}\int_{M}Rdvol_{g(t)}\leq (1-\frac{n}{2})\int_{M}R^2dvol_{g(t)}\leq 0.
\end{align}
Then $R(g(t)) \in L^1(M)$  for $0\leq t< t_{max}$.
Hence the ADM mass  is well-defined on $[0,t_{max})$ by Theorem \ref{A}, Theorem \ref{Bartnik} and Corollary \ref{positive}.
$\Box$

\section{Weighted Spaces for parabolic operators on Asymptotically flat manifolds}\label{section_5}
The theory of  the weighted spaces for elliptic operators on asymptotically flat manifolds  was first introduced by Nirenberg
and Walker \cite{NW}, and has been studied by many mathematicians such as Lockhart \cite{Lo1},
McOwen \cite{M}, Cantor \cite{C}, Bartnik \cite{B} and others.
In the section 1 of \cite{B}, R.Bartnik obtained the global estimates of the weighted spaces for elliptic operators (see Definition \ref{elliptic_wss}) close to Laplacian by using scaling arguments. Inspired by Bartnik's work, we introduce the following definitions of weighted spaces for parabolic operators on asymptotically flat manifolds.

\begin{defn}\label{parabolic_wss}
Suppose $(M^n,g)$ is an $n$-dimensional asymptotically flat manifold with asymptotic coordinates $\{x^i\}$. Denote $D^j_x v=\sup\limits_{|\alpha|=j}|\frac{\partial^{|\alpha|}}{\partial x_{i_1}\cdots\partial x_{i_j}}v|$.
Let $r(x)=|x|$ on $M_{\infty}$ and extend $r$ to a smooth positive function on all of $M^n$.
For parabolic domain $Q_T=M^n\times [0,T]$, $q\geq 1$ and $\beta\in\mathbb{R}$,
the weighted Lebesgue space $L^q_{\beta}(Q_T)$ is defined as the set of locally integrable functions $v$ for which the norm
\begin{align}\label{eq_aaaaaaaaaaaaaaa}
 ||v||_{L^q_\beta(Q_T)}=\left\{
                      \begin{array}{ll}
                        (\int^T_0\int_{M}|v|^q r^{-\beta q-n}dxdt)^{\frac{1}{q}}, & \hbox{$q<\infty$;} \\
                        ess \sup\limits_{Q_T} (r^{-\beta}|v|), & \hbox{$q=\infty$.}
                      \end{array}
                    \right.
\end{align}
is finite. For an nonnegative even integer $k$,  the weighted Sobolev space $W^{k,k/2,q}_\beta(Q_T)$ is defined as the set of functions $v$ for which the norm
$$
||v||_{W^{k,k/2,q}_\beta(Q_T)}=\sum\limits_{i+2j\leq k}||D^i_xD^j_t v||_{L^q_{\beta-i-2j}(Q_T)}.
$$
is finite.
For a nonnegative integer $k$, the weighted $C^k$ space $C^k_{\beta}(Q_T)$ is defined as the set of $C^k$ functions $v$ for which the norm
$$
||v||_{C^k_\beta(Q_T)}=\sum\limits_{i+2j\leq k}\sup\limits_{Q_T} r^{-\beta+i+2j}|D^i_x D^j_tv|
$$
is finite. Moreover, we define
\begin{align*}
&\ \ [v]_{C^{k+\alpha}_\beta(Q_T)}\\
&=\sum\limits_{i+2j= k} \sup\limits_{(x,t)\neq (y,s)\in Q_T}\min(r(x),r(y))^{-\beta+i+2j+\alpha}\frac{|D^i_xD^j_tv(x,t)-D^i_xD^j_tv(y,s)|}{\delta((x,t),(y,s))^{\alpha}},
\end{align*}
where $\delta((x,t),(y,s))=|x-y|+|t-s|^{\frac{1}{2}}$
and
\begin{align*}
&\ \ <v>_{C^{k+\alpha}_\beta(Q_T)}\\
=&\sum\limits_{i+2j= k-1} \sup\limits_{(x,t)\neq (y,s)\in Q_T}r(x)^{-\beta+i+2j+\alpha+1}\frac{|D^i_xD^j_tv(x,t)-D^i_xD^j_tv(x,s)|}{|t-s|^{\frac{\alpha+1}{2}}}
\end{align*}
for $k\geq 1$.
Then the weighted H\"{o}lder space $C^{k+\alpha,(k+\alpha)/2}_{\beta}(Q_T)$ is defined as the set of functions $v$ for which the norm
$$
||v||_{C^{k+\alpha,(k+\alpha)/2}_\beta(Q_T)}=||v||_{C^{k}_\beta(Q_T)}+[v]_{C^{k+\alpha}_\beta(Q_T)}+<v>_{C^{k+\alpha}_\beta(Q_T)}
$$
is finite.
\end{defn}

\begin{rem}\label{parabolic_wss_rem}
Consider the rescaled function
$v_R(x,t)=v(Rx,R^2t)$, where $R$ denotes a positive constant.
Let $y=Rx$, $\bar{t}=R^2 t$. Denote $A_r=B(o,2R)\backslash B(o,R)$ be the annulus on $\mathbb{R}^n$ and $\Delta_0$ be the stand Laplacian with flat metric on $\mathbb{R}^n$.
By a simple change of variables, we have
$$||D_x^j v_R||_{L^p_{\beta-j}(A_1\times [0,T])}=R^{\beta-\frac{2}{p}}||D_x^jv||_{L^p_{\beta-j}(A_R\times [0,R^2T])},$$
$$||\frac{\partial}{\partial t} v_R||_{L^p_{\beta-2}(A_1\times [0,T])}=R^{\beta-\frac{2}{p}}||\frac{\partial}{\partial \bar{t}}v||_{L^p_{\beta-2}(A_R\times [0,R^2T])},
$$
and
$$
||v_R||_{C^{k+\alpha,(k+\alpha)/2}_\beta(A_1\times [0,T])}=R^{\beta}||v||_{C^{k+\alpha,(k+\alpha)/2}_\beta(A_R\times [0,R^2T])}.
$$
\end{rem}

We also use the following different weighted spaces for parabolic operators which have no weights on the time-derivative terms.

\begin{defn}\label{parabolic_wss_1}
Suppose $(M^n,g)$ is an $n$-dimensional asymptotically flat manifold with asymptotic coordinates $\{x^i\}$. Denote $D^j_x v=\sup\limits_{|\alpha|=j}|\frac{\partial^{|\alpha|}}{\partial x_{i_1}\cdots\partial x_{i_j}}v|$.
Let $r(x)=|x|$ on $M_{\infty}$ and extend $r$ to a smooth positive function on all of $M^n$.
For parabolic domain $Q_T=M^n\times [0,T]$, $q\geq 1$ and $\beta\in\mathbb{R}$,
we also define weighted Sobolev space $\widetilde{W}^{k,k/2,q}_\beta(Q_T)$ as the set of functions $v$ for which the norm
$$
||v||_{\widetilde{W}^{k,k/2,q}_\beta(Q_T)}=\sum\limits_{i+2j\leq k}||D^i_xD^j_t v||_{L^q_{\beta-i}(Q_T)}.
$$
is finite, where $|| \cdot||_{L^q_{\beta-i}(Q_T)}$ is the norm defined in (\ref{eq_aaaaaaaaaaaaaaa}).
 The weighted $\widetilde{C}^k$ space $\widetilde{C}^k_{\beta}(Q_T)$ is defined as the set of $C^k$ functions $v$ for which the norm
$$
||v||_{\widetilde{C}^k_\beta(Q_T)}=\sum\limits_{i+2j\leq k}\sup\limits_{Q_T} r^{-\beta+i}|D^i_x D^j_tv|
$$
is finite. Moreover, we define
\begin{align*}
&\ \ [v]_{\widetilde{C}^{k+\alpha}_\beta(Q_T)}\\
&=\sum\limits_{i+2j= k} \sup\limits_{(x,t)\neq (y,t)\in Q_T}\min(r(x),r(y))^{-\beta+i+\alpha}\frac{|D^i_xD^j_tv(x,t)-D^i_xD^j_tv(y,t)|}{|x-y|^{\alpha}},
\end{align*}
and
\begin{align*}
&\ \ <v>_{\widetilde{C}^{k+\alpha}_\beta(Q_T)}\\
=&\sum\limits_{i+2j= k-1} \sup\limits_{(x,t)\neq (x,s)\in Q_T}r(x)^{-\beta+i}\frac{|D^i_xD^j_tv(x,t)-D^i_xD^j_tv(x,s)|}{|t-s|^{\frac{\alpha+1}{2}}}
\end{align*}
where $k\geq 1$.
Then the weighted H\"{o}lder space $\widetilde{C}^{k+\alpha,(k+\alpha)/2}_{\beta}(Q_T)$ is defined as the set of functions $v$ for which the norm
$$
||v||_{\widetilde{C}^{k+\alpha,(k+\alpha)/2}_\beta(Q_T)}=||v||_{\widetilde{C}^{k}_\beta(Q_T)}+[v]_{\widetilde{C}^{k+\alpha}_\beta(Q_T)}+<v>_{\widetilde{C}^{k+\alpha}_\beta(Q_T)}
$$
is finite.
\end{defn}
\begin{rem}\label{parabolic_wss_rem1}
(i)
 Consider the rescaled function
$v_R(x,t)=v(Rx,t)$.
Let $y=Rx$. Let $A_R$ and $\Delta_0$ be the notations defined in Remark \ref{parabolic_wss_rem}.
By a simple change of variables, we have
$$||D_x^j v_R||_{L^p_{\beta-j}(A_1\times [0,T])}=R^{\beta}||D_x^jv||_{L^p_{\beta-j}(A_R\times [0,T])},$$
$$||D_t v_R||_{L^p_{\beta}(A_1\times [0,T])}=R^{\beta}||D_tv||_{L^p_{\beta}(A_R\times [0,T])},$$
and
$$||v_R||_{\widetilde{C}^{k+\alpha,(k+\alpha)/2}_\beta(A_1\times [0,T])}=R^{\beta}||v||_{\widetilde{C}^{k+\alpha,(k+\alpha)/2}_\beta(A_R\times [0,T])}.$$

(ii) By Theorem \ref{Sobolev_embedding} (i) below, we have $||v||_{\widetilde{W}^{k,k/2,q}_\beta(Q_T)}\leq C||v||_{W^{k,k/2,q}_\beta(Q_T)}$.
\end{rem}

Then we also have following inequalities which related to the weighted spaces defined in Definition \ref{parabolic_wss} and Definition \ref{parabolic_wss_1}. The main idea for the proofs of these inequalities is the scaling arguments and covert local estimates to global estimates, which have been used by R.Bartnik \cite{B} to prove similar results for the weighted spaces for elliptic operators.
\begin{thm}\label{Sobolev_embedding}
Suppose $(M^n,g)$ is an $n$-dimensional asymptotically flat manifold with asymptotic coordinates $\{x^i\}$. Set $Q_T=M^n\times [0,T]$. Then the following inequalities hold:

(i) For $1\leq p\leq q\leq \infty$, $\beta_2<\beta_1$, we have
\begin{align}\label{holder_aaaaaaaaaaaaaa}
||v||_{L^p_{\beta_1}(Q_T)}\leq C||v||_{L^q_{\beta_2}(Q_T)}.
\end{align}

(ii) For $\beta=\beta_1+\beta_2$, $1\leq p,q,s\leq \infty$, $\frac{1}{p}=\frac{1}{q}+\frac{1}{s}$, we have
\begin{align}\label{holder_ineq}
||v||_{L^p_{\beta}(Q_T)}\leq ||v||_{L^q_{\beta_1}(Q_T)}|v||_{L^s_{\beta_2}(Q_T)},
\end{align}
and
\begin{align}\label{holder_ineq_2}
||v||_{C^{\alpha,\alpha/2}_{\beta}(Q_T)}\leq ||v||_{C^{\alpha,\alpha/2}_{\beta_1}(Q_T)}|v||_{C^{\alpha,\alpha/2}_{\beta_2}(Q_T)},
\end{align}
here $||v||_{C^{\alpha,\alpha/2}_{\beta}(Q_T)}=||v||_{C^{0}_\beta(Q_T)}+[v]_{C^{\alpha}_\beta(Q_T)}$ as we defined in Definition \ref{parabolic_wss}.

(iii)(Sobolev inequalities) For  $q\geq 1$, $n\geq 2$, we have
\begin{align}\label{Weighted_ss_4}
||D_xv||_{L^{\frac{(n+2)p}{n+2-p}}_{\beta-1}(Q_T)}\leq C ||v||_{\widetilde{W}^{2,1,q}_\beta(Q_T)}
\end{align}
if  $p<n+2$ and $p\leq q\leq \frac{(n+2)p}{n+2-p}$,
\begin{align}\label{Weighted_ss_2}
||v||_{L^{\frac{(n+2)p}{n+2-2p}}_{\beta}(Q_T)}\leq C ||v||_{\widetilde{W}^{2,1,q}_\beta(Q_T)}
\end{align}
if  $p<\frac{n+2}{2}$ and $p\leq q\leq \frac{(n+2)p}{n+2-2p}$,
and
\begin{align}\label{Weighted_ss_3}
||v||_{\widetilde{C}_{\beta}^{m,m/2}(Q_T)}\leq C ||v||_{\widetilde{W}^{2,1,p}_\beta(Q_T)}
\end{align}
if $p>n+2$, $m=2-\frac{n+2}{p}$.
\end{thm}

\begin{proof}
(i) and (ii) follow from Theorem 1.2 in \cite{B} directly.
Recall a Riemannian manifold $M^n$ with $C^{\infty}$ metric $g$ is called asymptotically flat of order $\tau>0$
if there exists a decomposition $M^n= M_0\cup M_{\infty}$
 with $M_0$ compact and a diffeomorphism $\phi:M_{\infty}\to\mathbb{R}^n-B(o,R_0)$ for some $R_0 > 0$ satisfying
(\ref{AE}). We denote $A_R=B(o,2R)\backslash B(o,R)$ be the annulus on $\mathbb{R}^n$ and $E_R=\mathbb{R}^n\backslash B(o,R)$.
Consider the rescaled function
\begin{align*}
    v_R(x,t)=v(Rx,t).
\end{align*}
Let $y=Rx$.
It follows from Remark \ref{parabolic_wss_rem1} (i) that
\begin{align*}
||v_R(x,t)||_{\widetilde{W}^{2,1,p}_\beta(A_1\times [0,T])}=R^{\beta}||v(y,t)||_{\widetilde{W}^{2,1,p}_\beta(A_R\times [0,T])},
\end{align*}
Set $p^*=\frac{(n+2)p}{n+2-p}$. Note that the weighted Sobolev spaces defined in Definition \ref{parabolic_wss} and Definition \ref{parabolic_wss_1} are equivalent to the usual Sobolev spaces on bounded domains.
Then
\begin{align*}
||D_xv||_{L^{p^*}_{\beta-1}(A_R\times [0,T])}
&= R^{-\beta}||D_xv_R||_{L^{p^*}_{\beta-1}(A_1\times [0,T])}\\
&\leq C R^{-\beta}||v_R(x,t)||_{\widetilde{W}^{2,1,q}_\beta(A_1\times [0,T])}\\
&= C ||v(x,t)||_{\widetilde{W}^{2,1,q}_\beta(A_R\times [0,T])},
\end{align*}
by using the usual Sobolev inequality to $A_1$ (\cite{LSN}, p.80, Lemma 3.3 or \cite{Ch}, p.29,
Theorem 2.3 ), where $C$ depends on $n$, $p$, $A_1$ and $T^{-1}$.
We write $v=\sum\limits_{j=0}^{\infty}v_j$ with $v_0=v|_{M_0}$ and $v_j=v|_{A_{2^{j-1}R_0 }}$ for $j\geq 1$. Then
\begin{align*}
||D_xv||_{L^{p^*}_{\beta-1}(Q_T)}&=(||D_xv_0||^{p^*}_{L^{p^*}_{\beta-1}(M_0\times[0,T])}+\sum\limits_{j=1}^{\infty}||D_xv_j||^{p^*}_{L^{p^*}_{\beta-1}(A_{2^{j-1}R_0 }\times [0,T])})^{\frac{1}{p^*}}\\
&\leq C(||v_0||_{\widetilde{W}^{2,1,q}_\beta(M_0\times[0,T])}^{p^*}+\sum\limits_{j=1}^{\infty}||v_j||_{\widetilde{W}^{2,1,q}_\beta(A_{2^{j-1}R_0 }\times [0,T])}^{p^*})^{\frac{1}{^{p^*}}}\\
&\leq  C(||v_0||_{\widetilde{W}^{2,1,q}_\beta(M_0\times[0,T])}^{q}+\sum\limits_{j=1}^{\infty}||v_j||_{\widetilde{W}^{2,1,q}_\beta(A_{2^{j-1}R_0 }\times [0,T])}^{q})^{\frac{1}{^{q}}},
\end{align*}
since $p^*\geq q$ and $(\sum a_j^{p^*})^{\frac{1}{p^*}}\leq (\sum a_j^q)^{\frac{1}{q}}$ for $a_j\geq 0$. Therefore, (\ref{Weighted_ss_4}) holds clearly.
Since we have
\begin{align}\label{scaling_property}
||v_R||_{\widetilde{C}^{k+\alpha,(k+\alpha)/2}_\beta(A_1\times [0,T])}=R^{\beta}||v||_{\widetilde{C}^{k+\alpha,(k+\alpha)/2}_\beta(A_R\times [0,T])},
\end{align}
 (\ref{Weighted_ss_2}) and (\ref{Weighted_ss_3}) follow from same rescaling arguments and usual Sobolev inequalities (\cite{LSN}, p.80, Lemma 3.3 or \cite{Ch}, p.29, Theorem 2.3 and p.38, Theorem 3.4).
\end{proof}

\section{weighted estimates for Yamabe flow}\label{section_Y_w}

In this section, we give the proof of Theorem \ref{A}. Suppose $u(x,t)$, $0\leq t<t_{max}$, be the fine solution to the Yamabe flow (\ref{yamabe_flow_u}) on an $n$-dimensional asymptotically flat manifold $(M^n,g_0)$. Let $Q_T=M\times [0,T]$, $T<t_{max}$.  The key part is to prove $1-u(x,t)\in C^{2+\alpha,1+\alpha/2}_{-\tau}(Q_T)$ (this implies $g_{ij}(x,t)-\delta_{ij}\in C^{2+\alpha}_{-\tau}(M)$ for $0\leq t<t_{max}$). Using the scaling arguments used in section \ref{section_5}, we solve this problem by the global $L^p$ and Schauder estimates to Yamabe flow for weighted spaces defined in section \ref{section_5}.

First, the following theorem
shows that $1-u(x,t)$ belongs to weighted space $C^{0}_{-(\tau+2)}(Q_T)$ for any $T<t_{max}$.

\begin{thm}\label{u_decay}
Let $u(x,t)$, $0\leq t<t_{max}$, be the fine solution to the Yamabe flow (\ref{yamabe_flow_u}) on an $n$-dimensional asymptotically flat manifold $(M^n,g_0)$ of order $\tau>0$. Set $v=1-u$. Then
\begin{align}\label{L0_decay}
v(x,t)=O(r^{-(\tau+2)}).
\end{align}
for all $t\in [0,t_{max})$ and $v(x,t)\in C^{0}_{-(\tau+2)}(Q_T)$ for any $T<t_{max}$.
\end{thm}
\begin{proof}
Setting $v=1-u$, we have
\begin{align}\label{v_equation}
    N(1-v)^{N-1}v_t=\Delta_{g_0}v-aR_{g_0}v+aR_{g_0}.
\end{align}
 Set $h(x)=r(x)^{\tau+2}$ and $w=h(x)v(x,t)$, where $r(x)$ is the function defined in Definition \ref{elliptic_wss}. Then, by direct computation, we have
\begin{align*}
  N(1-v)^{N-1}w_t=&\Delta_{g_0} w-2\nabla_{g_0}\log h\cdot\nabla_{g_0}w\\
-&w(\frac{\Delta_{g_0} h}{h}-2\frac{|\nabla_{g_0} h|^2}{h^2}+aR_{g_0})+ahR_{g_0}.
\end{align*}
Hence,
\begin{align*}
 w_t&=div_{g_0}(\frac{1}{ N(1-v)^{N-1}}\nabla_{g_0}w) +b\cdot\nabla_{g_0}w\\
&\ \ \ +dw+\frac{a}{ N(1-v)^{N-1}}hR_{g_0},
\end{align*}
where $b(w,t)=-(\nabla_{g_0}(\frac{1}{ N(1-v)^{N-1}})+2\nabla_{g_0}\log h)$ and $d(x,t)=-\frac{1}{ N(1-v)^{N-1}}(\frac{\Delta_{g_0} h}{h}-2\frac{|\nabla_{g_0} h|^2}{h^2}+aR_{g_0})$.
Note that $R_{g_0}
=\partial_j(\partial_i g_{ij}- \partial_j g_{ii})+O(r^{-2\tau-2})=O(r^{-(\tau+2)})$. Therefore,
$|\frac{a}{ N(1-v)^{N-1}}hR_{g_0}|\leq C$ and $ |d|\leq D$ for $t\in[0,T]$ by the definitions of $h$ and fine solution.
Set $\widetilde{w}=w-e^{(D+DCT)t}-Ct$. Then we have
\begin{align*}
 \widetilde{w}_t&\leq div_{g_0}(\frac{1}{ N(1-v)^{N-1}}\nabla_{g_0}\widetilde{w}) +b\cdot\nabla_{g_0}\widetilde{w}\\
&\ \ \ \ +dw-De^{(D+DCT)t}-DCTe^{(D+DCT)t}\\
&\leq div_{g_0}(\frac{1}{ N(1-v)^{N-1}}\nabla_{g_0}\widetilde{w}) +b\cdot\nabla_{g_0}\widetilde{w}+d\widetilde{w}.
\end{align*}
Then (\ref{L0_decay}) follows from Theorem \ref{Ecker_Huisken} immediately.
\end{proof}

Second, we give the global $L^p$  and Schauder estimates to the Yamabe flow for the weighted spaces defined in section \ref{section_5}.
\begin{thm}\label{Lp_estimate_II}
Let $u(x,t)$, $0\leq t<t_{max}$, be the fine solution to the Yamabe flow (\ref{yamabe_flow_u}) on an $n$-dimensional asymptotically flat manifold $(M^n,g_0)$ of order $\tau>0$. Assume $0<\delta\leq u(x,t)\leq C'$ on $[0,T]$, $T<t_{max}$.
Set $v=1-u$. Then there is a constant $C=C(n,p,\tau,\beta,\delta, C')$ such that
\begin{align}\label{Lp_estimate_II_1}
||v||_{W^{2,1,p}_\beta (Q_T)}\leq C(||R_{g_0}||_{L^p_{\beta-2}(Q_T)}+||v||_{L^p_{\beta}(Q_T)}).
\end{align}
\end{thm}

\begin{proof}
Recall $v=1-u$ evolves as
\begin{align}\label{equation_111}
    N(1-v)^{N-1}v_t=\Delta_{g_0}v-aR_{g_0}v+aR_{g_0}
\end{align}
along the Yamabe flow (\ref{yamabe_flow_u}).
We change time by a constant scale
such that (\ref{equation_111}) is equivalent to the following equation
\begin{align}
    (1-v)^{N-1}v_t=\Delta_{g_0}v-aR_{g_0}v+aR_{g_0}.
\end{align}
Let $A_R=B(o,2R)\backslash B(o,R)$ in $\mathbb{R}^n$ and $E_R=\mathbb{R}^n\backslash B(o,R)$. Denote $\Delta_0$ be the stand Laplacian  with flat metric. Here we consider the rescaled function
\begin{align*}
    v_R(x,t)=v(Rx,R^2t).
\end{align*}
Let $y=Rx$, $\bar{t}=R^2 t$.
It follows from Remark \ref{parabolic_wss_rem} that
\begin{align*}
||v_R||_{W^{2,1,p}_\beta(A_1\times [0,T])}=R^{\beta-\frac{2}{p}}||v||_{W^{2,1,p}_\beta(A_R\times [0,R^2T])},
\end{align*}
and
\begin{align*}
||(\frac{\partial}{\partial t}-\Delta_{0}) v_R||_{L^p_{\beta-2}(A_1\times [0,T])}=R^{\beta-\frac{2}{p}}||(\frac{\partial}{\partial \bar{t}}-\Delta_{0})v||_{L^p_{\beta-2}(A_R\times [0,R^2T])}.
\end{align*}
We write $v=\sum\limits_{j=0}^{\infty}v_j$ with $v_0=v|_{M_0}$ and  $v_j=v|_{A_{2^{j-1}R  }}$ for $j\geq 1$, where $R\geq R_0$ and $R_0$ is the constant defined in Definition \ref{AE_def}. Note that $v_j$ vanishes outside $A_{2^{j-1}R}$.
By the standard $L^p$ estimates for parabolic equation,
\begin{align*}
&\qquad||v_j||_{W^{2,1,p}_\beta(A_{2^{j-1}R  }\times [0,T])}
= (2^{j-1}R  )^{-(\beta-\frac{2}{p})}||(v_j)_{2^{j-1}R  }||_{W^{2,1,p}_\beta(A_1\times [0,(2^{j-1}R  )^{-2}T])}\\
&\leq C (2^{j-1}R  )^{-(\beta-\frac{2}{p})}(||(\partial_t-\Delta_{0})(v_j)_{2^{j-1}R  }||_{L^p_{\beta-2}(A_1\times [0,(2^{j-1}R )^{-2}T])}\\
&\qquad +||(v_j)_{2^{j-1}R  }||_{L^p_{\beta}(A_1\times [0,(2^{j-1}R  )^{-2}T])})\\
&=C (||(\partial_t-\Delta_{0})v_j||_{L^p_{\beta-2}(A_{2^{j-1}R  }\times [0,T])}+||v_j||_{L^p_{\beta}(A_{2^{j-1}R  }\times [0,T])}),
\end{align*}
where $C$ is independent of $j$ and $T$ (see Proposition 7.11 in \cite{L}). Therefore,
\begin{align}\label{Lp_estimate_f_1}
&||v||_{W^{2,1,p}_\beta(E_{R}\times[0,T])}=(\sum\limits_{j=1}^{\infty}||v_j||^p_{W^{2,1,p}_\beta(A_{2^{j-1}R   }\times [0,T])})^{\frac{1}{p}}\nonumber\\
\leq &C (\sum\limits_{j=1}^{\infty}(||(\partial_t-\Delta_{0})v_j||_{L^p_{\beta-2}(A_{2^{j-1}R}\times [0,T])}+||v_j||_{L^p_{\beta}(A_{2^{j-1}R   }\times [0,T])})^p)^{\frac{1}{p}}\nonumber\\
\leq &C(||(\partial_t-\Delta_{0})v||_{L^p_{\beta-2}(E_{R   }\times[0,T])}+||v||_{L^p_{\beta}(E_{R   }\times[0,T])}).
\end{align}
Set $Pv=h(\Delta_{g_0}v-aR_{g_0}v)=h(g^{ij}_0\frac{\partial^2 v}{\partial x^i \partial x^j}+b^j\frac{\partial v}{\partial x^j}-aR_{g_0}v)$, where $h=\frac{1}{ (1-v)^{N-1}}$ and
$b^j=\frac{1}{\sqrt{\text{det} g_0}}\frac{\partial}{\partial x_i}(\sqrt{\text{det} g_0}g^{ij}_0)$.
If $supp(v)\subset E_R$, then
\begin{align*}
&\quad||(\Delta_{0}-P)v||_{L^p_{\beta-2}(Q_T)}\\
&\leq \sup\limits_{|x|>R}|hg^{ij}_0-\delta_{ij}|||D_x^2v||_{L^p_{\beta-2}(Q_T)}+||hb^j\frac{\partial v}{\partial x^j}||_{L^p_{\beta-2}(Q_T)}\\
&\ \ \ +||a hR_{g_0}v||_{L^p_{\beta-2}(Q_T)}\\
&\leq \sup\limits_{|x|>R}|hg^{ij}_0-\delta_{ij}|||D_x^2v||_{L^p_{\beta-2}(Q_T)}+C||b||_{L^{q}_{-1}(E_R\times [0,T])}||D_xv||_{L^{\frac{pq}{q-p}}_{\beta-1}(Q_T)}\\
&\ \ \ +C||R_{g_0}||_{L^{\frac{q}{2}}_{-2}(E_R\times [0,T])}||v||_{L^{\frac{pq}{q-2p}}_{\beta}(Q_T)}\\
&\leq \sup\limits_{|x|>R}|hg^{ij}_0-\delta_{ij}|||D_x^2v||_{L^p_{\beta-2}(Q_T)}+C||b||_{L^{q}_{-1}(E_R\times [0,T])}||v||_{W^{2,1,p}_{\beta}(Q_T)}\\
&\ \ \ +C||R_{g_0}||_{L^{\frac{q}{2}}_{-2}(E_R\times [0,T])}||v||_{W^{2,1,p}_{\beta}(Q_T)}\\
&\leq (\sup\limits_{|x|>R}|hg^{ij}_0-\delta_{ij}|+C||b||_{L^{q}_{-1}(E_R\times [0,T])}+C||R_{g_0}||_{L^{\frac{q}{2}}_{-2}(E_R\times [0,T])})||v||_{W^{2,1,p}_{\beta}(Q_T)},
\end{align*}
by using Theorem \ref{Sobolev_embedding}.
Note that
\begin{align}\label{Lp_estimate_II_3}
 \sup\limits_{|x|>R}|hg^{ij}_0-\delta_{ij}|+C||b||_{L^{q}_{-1}(E_R\times [0,T])}+C||R_{g_0}||_{L^{\frac{q}{2}}_{-2}(E_R\times [0,T])} \to 0,
\end{align}
as $R\to\infty$, by the asymptotic condition (\ref{AE}) and Theorem \ref{Sobolev_embedding} (i).
Using (\ref{Lp_estimate_f_1}), (\ref{Lp_estimate_II_3}), we obtain
\begin{align}\label{Lp_estimate_f_5}
||v_{\infty}||_{W^{2,1,p}_\beta (Q_T)}
\leq C(||(\partial_t-\Delta_{g_0})v_{\infty}||_{L^p_{\beta-2}(Q_T)}+||v_{\infty}||_{L^p_{\beta}(Q_T)})
\end{align}
for $R$ sufficient large. Then
\begin{align}\label{Lp_estimate_f_6}
&||(\partial_t-P)v_{\infty}||_{L^p_{\beta-2}(Q_T)}\nonumber\\
\leq &  ||(\partial_t-P)v||_{L^p_{\beta-2}(Q_T)}+||(\partial_t-P)(\zeta_R v)||_{L^p_{\beta-2}(Q_T)}\nonumber\\
\leq & 2 ||(\partial_t-P)v||_{L^p_{\beta-2}(Q_T)}+||hv\Delta_{g_0}\zeta_R+2h\nabla_{g_0}u\cdot\nabla_{g_0}\zeta_R||_{L^p_{\beta-2}(A_R\times[0,T])}\nonumber\\
\leq & 2 ||(\partial_t-P)v||_{L^p_{\beta-2}(Q_T)}+C||v+|\nabla_{g_0}u|||_{L^p(A_R\times[0,T])}.
\end{align}
Similar to (\ref{Lp_estimate_f_6}),
\begin{align}\label{Lp_estimate_f_7}
&||(\partial_t-P)v_{0}||_{L^p_{\beta-2}(Q_T)}\nonumber\\
\leq &  ||(\partial_t-P)v||_{L^p_{\beta-2}(Q_T)}+C||v+|\nabla_{g_0}u|||_{L^p(A_R\times[0,T])}.
\end{align}
By using the interpolation inequality, standard parabolic $L^p$ estimate on $v_0$, (\ref{Lp_estimate_f_5}), (\ref{Lp_estimate_f_6}) and (\ref{Lp_estimate_f_7}), we conclude that
\begin{align*}
||v||_{W^{2,1,p}_\beta (Q_T)}\leq C(||(\partial_t-P)v||_{L^p_{\beta-2}(Q_T)}+||v||_{L^p_{\beta}(Q_T)}).
\end{align*}
Then Theorem \ref{Lp_estimate_II}
holds immediately.
\end{proof}

\begin{thm}\label{Holder_estimate_II_thm}
Let $u(x,t)$, $0\leq t<t_{max}$, be the fine solution to the Yamabe flow (\ref{yamabe_flow_u}) on an $n$-dimensional asymptotically flat manifold $(M^n,g_0)$ of order $\tau>0$. Set $v=1-u$. Assume $0<\delta\leq u(x,t)\leq C'$ on $[0,T]$, $||v||_{C^{\alpha,\alpha/2}_0(Q_T)}\leq C''$ and $R_{g_0}\in C^{\alpha}_{-2-\tau}(M)$, where $T<t_{max}$.
 Then there is a constant $C=C(n,p,\tau,\beta,\delta, C',C'')$ such that
\begin{align}\label{Holder_estimate_II}
||v||_{C^{2+\alpha,1+\alpha/2}_\beta (Q_T)}\leq C(||R_{g_0}||_{C^{\alpha,\alpha/2}_{\beta-2} (Q_T)}+||v||_{C^0_{\beta}(Q_T)}).
\end{align}
\end{thm}
\begin{proof}
We use the same notations as the proof of Theorem \ref{Lp_estimate_II}.
 We have
$
||v_R||_{C^{k+\alpha,(k+\alpha)/2}_\beta(A_1\times [0,T])}=R^{\beta}||v||_{C^{k+\alpha,(k+\alpha)/2}_\beta(A_R\times [0,R^2T])}
$
by Remark \ref{parabolic_wss_rem}.
Using the similar scaling arguments to Theorem \ref{Lp_estimate_II} and the standard Schauder estimate for parabolic equations (see Theorem 4.9 in \cite{L}), we can obtain
\begin{align}\label{Holder_estimate_f_1}
&||v||_{C^{2+\alpha,1+\alpha/2}_\beta(E_{R}\times[0,T])}
\leq &C(||(\partial_t-\Delta_{0})v||_{C^{\alpha,\alpha/2}_{\beta-2}(E_{R   }\times[0,T])}+||v||_{C^0_{\beta}(E_{R   }\times[0,T])})
\end{align}
for $R\geq R_0$, where $R_0$ is the constant defined in Definition \ref{AE_def}.
 Since $0<\delta\leq u(x,t)\leq C'$ on $[0,T]$ and $||v||_{C^{\alpha,\alpha/2}_0(Q_T)}\leq C''$, $||h||_{C^{\alpha,\alpha/2}_{0}(Q_T)}\leq C'''$, where $C'''$ is a constant only depending on $\delta,N,C',C''$.
If $supp(v)\subset E_R$, we compute
\begin{align}\label{Holder_estimate_f_2}
&\quad||(\Delta_{0}-P)v||_{C^{\alpha,\alpha/2}_{\beta-2}(Q_T)}\nonumber\\
&\leq  ||hg^{ij}_0-\delta_{ij}||_{C^{\alpha,\alpha/2}_{0}(E_R\times [0,T])}||D_x^2v||_{C^{\alpha,\alpha/2}_{\beta-2}(Q_T)}\nonumber\\
&\quad+||h||_{C^{\alpha,\alpha/2}_{0}(E_R\times [0,T])}||b||_{C^{\alpha,\alpha/2}_{-1}(E_R\times [0,T])}||D_xv||_{C^{\alpha,\alpha/2}_{\beta-1}(Q_T)}\nonumber\\
&\quad+||h||_{C^{\alpha,\alpha/2}_{0}(E_R\times [0,T])}||R_{g_0}||_{C^{\alpha,\alpha/2}_{-2}(E_R\times [0,T])}||v||_{C^{\alpha,\alpha/2}_{\beta}(Q_T)}\nonumber\\
&\leq  C(||hg^{ij}_0-\delta_{ij}||_{C^{\alpha,\alpha/2}_{0}(E_R\times [0,T])}
+||b||_{C^{\alpha,\alpha/2}_{-1}(E_R\times [0,T])}\nonumber\\
&\quad+||R_{g_0}||_{C^{\alpha,\alpha/2}_{-2}(E_R\times [0,T])})||v||_{C^{2+\alpha,1+\alpha/2}_{\beta}(Q_T)}
\end{align}
by using Theorem \ref{Sobolev_embedding}.
Note that
\begin{align}\label{Holder_estimate_f_3}
||hg^{ij}_0-\delta_{ij}||_{C^{\alpha,\alpha/2}_{0}(E_R\times [0,T])}
+||b||_{C^{\alpha,\alpha/2}_{-1}(E_R\times [0,T])}+||R_{g_0}||_{C^{\alpha,\alpha/2}_{-2}(E_R\times [0,T])} \to 0,
\end{align}
as $R\to\infty$ by the asymptotic condition (\ref{AE}), $R_{g_0}\in C^{\alpha}_{-2-\tau}(M)$, and Theorem \ref{Sobolev_embedding}.
So by using the similar arguments of Theorem \ref{Sobolev_embedding}, we see that Theorem \ref{Holder_estimate_II_thm} holds.
\end{proof}

Finally, we give the proof of Theorem \ref{A}.

 \textbf{Proof of Theorem \ref{A}.}
 By Theorem \ref{u_decay}, we have
$v \in L^{\infty}_{-\tau-2}(Q_T)$. Then $R_{g_0}\in L^{p}_{-\tau-2+\epsilon}(Q_T)$ and $v\in L^{p}_{-\tau+\epsilon}(Q_T)$ for any $1<p<\infty$ and $\epsilon>0$ by (\ref{holder_aaaaaaaaaaaaaa}). It follows from Theorem \ref{Lp_estimate_II} that $v\in W_{-\tau+\epsilon}^{2,1,p}$ for any $1<p<\infty$ and $\epsilon>0$.
So $v\in \widetilde{C}^{1+\alpha,(1+\alpha)/2}_{-\tau+\epsilon}(Q_T)$ by (\ref{Weighted_ss_3}) and Remark \ref{parabolic_wss_rem1} (ii).
 Then
\begin{align*}
&\qquad||v||_{C^{\alpha,\alpha/2}_0(Q_T)}\\
&=||v||_{C^{\alpha}_0(Q_T)}+\sup\limits_{(x,t)\neq (y,s)\in Q_T}\min(r(x),r(y))^{\alpha}\frac{|v(x,t)-v(y,s)|}{\delta((x,t),(y,s))^{\alpha}}\\
&\leq C(||v||_{C^{\alpha}_0(Q_T)}+\sup\limits_{(x,t)\neq (y,s)\in Q_T}\min(r(x),r(y))^{\alpha}\frac{|v(x,t)-v(y,t)|}{|x-y|^{\alpha}}\\
&\qquad +\sup\limits_{(x,t)\neq (y,s)\in Q_T}\min(r(x),r(y))^{\alpha}\frac{|v(y,t)-v(y,s)|}{|t-s|^{\frac{\alpha}{2}}})\\
&\leq C||v||_{\widetilde{C}^{1+\alpha,(1+\alpha)/2}_{-\tau+\epsilon}(Q_T)}
\end{align*}
for $\epsilon>0$ small. Since $R_{g_0}\in C^{\alpha}_{-2-\tau}(M)$ and $R_{g_0}$ is independent of time, we have $R_{g_0}\in C^{\alpha,\alpha/2}_{-2-\tau} (Q_T)$.
By Theorem \ref{u_decay}, we know that
$v \in C^{0}_{-2-\tau}(Q_T)$. It follows from Theorem \ref{Holder_estimate_II_thm} that $v\in C^{2+\alpha,1+\alpha/2}_{-\tau}(Q_T)$.
It implies $1-u(x,t) \in C^{2+\alpha}_{-\tau}(M)$ for $0\leq t<t_{max}$ by the Definition \ref{parabolic_wss}.
$\Box$

\section{Appendix}
In this section, we give a proof of Theorem \ref{Ecker_Huisken}
 for sake of convenience for the readers.

 \textbf{Proof of Theorem \ref{Ecker_Huisken}:}
Define $\theta>0$ to be chosen
$$
h(y,t)=-\frac{\theta d^2_{g(t)}(p,y)}{4(2\eta-t)},0<t<\eta,
$$
where $d_{g(t)}(p,y)$ is the distance between $p$ and $y$ at time $t$ and
$0<\eta<\min(T,\frac{1}{64K},\frac{1}{32\alpha_4},\frac{1}{4\alpha_5})$. Then
\begin{align*}
\frac{d}{dt}h=-\frac{\theta d^2_{g(t)}(p,y)}{4(2\eta-t)^2}-\frac{\theta d_{g(t)}(p,y)}{2(2\eta-t)}\frac{d}{dt}d_{g(t)}(p,y).
\end{align*}
By (iv), we have
\begin{align*}
|\frac{d}{dt}d_{g(t)}(p,y)|\leq \frac{1}{2}\alpha_5 d_{g(t)}(p,y).
\end{align*}
Then we conclude that
\begin{align*}
    \frac{d}{dt}h\leq -\theta^{-1}|\nabla h|^2+\theta^{-1}\alpha_5|\nabla h|^2 (2\eta-t),
\end{align*}
We choose $\theta=\frac{1}{4\alpha_1}$,
then
\begin{align}\label{h_estiamte}
\frac{d}{dt}h+2a|\nabla h|^2\leq 0
\end{align}
by using $\eta\leq \frac{1}{4\alpha_5}$.
Taking $f_K=\max\{\min(f,K),0\}$ and $0<\epsilon<\eta$, we have
\begin{align*}
&\int^{\eta}_{\epsilon}e^{-\beta t}(\int_{M}\phi^2 e^h f_K(div(a\nabla f)-\frac{\partial f}{\partial t})d\mu_t)dt\\
\geq & -\alpha_2 \int^{\eta}_{\epsilon}e^{-\beta t}(\int_{M}\phi^2 e^h f_K|\nabla f|d\mu_t)dt\\
&-\alpha_3 \int^{\eta}_{\epsilon}e^{-\beta t}(\int_{M}\phi^2 e^h f_K fd\mu_t)dt
\end{align*}
for some smooth time independent compactly supported function $\phi$ on $M^n$, where $\beta>0$ will be chosen later. Then we have
\begin{align*}
0\leq &-\int^{\eta}_{\epsilon}e^{-\beta t}(\int_{M}\phi^2 e^h a <\nabla f_K,\nabla f_K>d\mu_t)dt\\
&-\int^{\eta}_{\epsilon}e^{-\beta t}(\int_{M}\phi^2 e^h f_K a<\nabla h,\nabla f>d\mu_t)dt\\
&-2\int^{\eta}_{\epsilon}e^{-\beta t}(\int_{M}\phi e^h f_K a<\nabla \phi,\nabla f>d\mu_t)dt\\
&-\int^{\eta}_{\epsilon}e^{-\beta t}(\int_{M}\phi^2 e^h f_K\frac{\partial f}{\partial t}d\mu_t)dt
+\alpha_3 \int^{\eta}_{\epsilon}e^{-\beta t}(\int_{M}\phi^2 e^h f_K fd\mu_t)dt \\
&+\alpha_2 \int^{\eta}_{\epsilon}e^{-\beta t}(\int_{M}\phi^2 e^h f_K|\nabla f|d\mu_t)dt\\
=&\textrm{I}+\textrm{II}+\textrm{III}+\textrm{IV}+\textrm{V}+\textrm{VI}.
\end{align*}
By Schwartz' inequality, we derive
\begin{align*}
\textrm{II}\leq \frac{1}{4}\int^{\eta}_{\epsilon}e^{-\beta t}(\int_{M}\phi^2 e^h a|\nabla f|^2d\mu_t)dt+\int^{\eta}_{\epsilon}e^{-\beta t}(\int_{M}\phi^2 e^h f_K^2 a|\nabla h|^2d\mu_t)dt,
\end{align*}
\begin{align*}
\textrm{III}\leq \frac{1}{2}\int^{\eta}_{\epsilon}e^{-\beta t}(\int_{M}\phi^2 e^h a|\nabla f|^2d\mu_t)dt+2\int^{\eta}_{\epsilon}e^{-\beta t}(\int_{M} e^h f_K^2 a|\nabla \phi|^2d\mu_t)dt,
\end{align*}
and
\begin{align*}
\textrm{VI}&\leq \frac{1}{4}\int^{\eta}_{\epsilon}e^{-\beta t}(\int_{M}\phi^2 e^h a|\nabla f|^2d\mu_t)dt+\alpha_2^2\int^{\eta}_{\epsilon}e^{-\beta t}(\int_{M} e^h f_K^2 \frac{1}{a}|\nabla \phi|^2d\mu_t)dt\\
&\leq \frac{1}{4}\int^{\eta}_{\epsilon}e^{-\beta t}(\int_{M}\phi^2 e^h a|\nabla f|^2d\mu_t)dt+\frac{\alpha_2^2}{\alpha_1'}\int^{\eta}_{\epsilon}e^{-\beta t}(\int_{M} e^h f_K^2 |\nabla \phi|^2d\mu_t)dt.
\end{align*}
Since
\begin{align*}
-e^hf_K\frac{\partial f}{\partial t}\leq -e^h f_K \frac{\partial f_K}{\partial t}+\frac{\partial }{\partial t}(e^hf_K(f_K-f)),
\end{align*}
and
$$f_K(f_K-f)\leq 0,$$
we obtain
\begin{align*}
&\qquad\textrm{IV}+\textrm{V}\\
&\leq -\frac{1}{2}\int^{\eta}_{\epsilon}e^{-\beta t}(\int_{M}\phi^2 e^h \frac{\partial f_K^2}{\partial t}d\mu_t)dt
+\int^{\eta}_{\epsilon}e^{-\beta t}(\int_{M}\phi^2 \frac{\partial }{\partial t}(e^h f_K(f_K-f))d\mu_t)dt\\
&
-\alpha_3 \int^{\eta}_{\epsilon}e^{-\beta t}(\int_{M}\phi^2 e^h f_K(f_K-f)d\mu_t)dt+\alpha_3 \int^{\eta}_{\epsilon}e^{-\beta t}(\int_{M}\phi^2 e^h f_K^2d\mu_t)dt.
\end{align*}
Moreover, we have
\begin{align*}
|\frac{d}{dt}(d\mu_t)|\leq n \alpha_5 d\mu_t
\end{align*}
by (iv). Now we choose $\beta\geq 2n\alpha_5+4\alpha_3+4\frac{\alpha_2^2}{\alpha_1'}$. Then
\begin{align*}
&\qquad\textrm{IV}+\textrm{V}\\
&\leq -\frac{1}{2}e^{-\beta t}\int_{M}\phi^2 e^h f_K^2d\mu_t|_{t=\eta}
+\frac{1}{2}e^{-\beta t}\int_{M}\phi^2 e^h f_K^2d\mu_t|_{t=\epsilon}\\
&+\frac{1}{2}\int^{\eta}_{\epsilon}e^{-\beta t}(\int_{M}\phi^2 e^h f_K^2 \frac{\partial h}{\partial t}d\mu_t)dt-\frac{1}{4}\beta\int^{\eta}_{\epsilon}e^{-\beta t}(\int_{M}\phi^2 e^h f_K^2 d\mu_t)dt\\
&
+e^{-\beta t}\int_{M}\phi^2 e^h f_K(f_K-f)d\mu_t|_{t=\eta}-e^{-\beta t}\int_{M}\phi^2 e^h f_K^2d\mu_t|_{t=\epsilon}.
\end{align*}
Combining the estimates of $\textrm{I}-\textrm{VI}$ and letting $\epsilon\to 0$, we get
\begin{align*}
&-\int^{\eta}_{0}e^{-\beta t}(\int_{M}\phi^2 e^h a |\nabla f_K|^2d\mu_t)dt
+\int^{\eta}_{0}e^{-\beta t}(\int_{M}\phi^2 e^h a |\nabla f|^2d\mu_t)dt\\
&+2\int^{\eta}_{0}e^{-\beta t}(\int_{M} e^h f_K^2 a|\nabla \phi|^2d\mu_t)dt-\frac{1}{2}e^{-\beta t}\int_{M}\phi^2 e^h f_K^2d\mu_t|_{t=\eta}\geq 0.
\end{align*}
by $f_K\equiv 0$ at $t=0$ and (\ref{h_estiamte}). Now we choose $0\leq \phi\leq 1$ satisfying $\phi \equiv 1$ on $B_{g_0}(p,R)$, $\phi \equiv 0$ outside $B_{g_0}(p,R+1)$ and $|\nabla_{g_0} \phi|_{g_0}\leq 2$. Then we have
\begin{align*}
&\frac{1}{2}e^{-\beta \eta}\int_{B_{g_0}(p,R)}\phi^2 e^h f_K^2d\mu_t|_{t=\eta}\leq \int^{\eta}_{0}e^{-\beta t}(\int_{B_{g_0}(p,R+1)}\phi^2 e^h a (|\nabla f|^2-|\nabla f_K|^2)d\mu_t)dt\\
&+C(\alpha_5)\int^{\eta}_{0}e^{-\beta t}(\int_{B_{g_0}(p,R+1)\backslash B_{g_0}(p,R)} e^h f_K^2 a d\mu_t)dt,
\end{align*}
where $C(\alpha_5)$ is a constant only depending on $\alpha_5$. By $0<\eta<\min(\frac{1}{K},\frac{1}{32\alpha_4})$ and volume growth assumptions
on $M^n$, we have
\begin{align*}
\int^{\eta}_{0}e^{-\beta t}(\int_{B_{g_0}(p,R+1)\backslash B_{g_0}(p,R)} e^h f_K^2 a d\mu_t)dt\to 0,
\end{align*}
as $R\to \infty$. Then we derive
\begin{align*}
&\frac{1}{2}e^{-\beta \eta}\int_{M}\phi^2 e^h f_K^2d\mu_t|_{t=\eta}\leq \int^{\eta}_{0}e^{-\beta t}(\int_{M}\phi^2 e^h a (|\nabla f|^2-|\nabla f_K|^2)d\mu_t)dt.
\end{align*}
Now letting $K\to\infty$, we conclude that
\begin{align*}
\frac{1}{2}e^{-\beta \eta}\int_{M}\phi^2 e^h (\max(f,0))^2d\mu_t|_{t=\eta}\leq 0,
\end{align*}
where $0<\eta<\min(T,\frac{1}{64K},\frac{1}{32\alpha_4},\frac{1}{4\alpha_5})$. By the inductive argument, we conclude that
$f\leq 0$ in $M^n\times[0,T]$.
$\Box$

\thanks{\textbf{Acknowledgement}: We would like thank  Prof. Li Ma bring
this problem to us. The first author would like thank Doctor Y.F.Chen for useful talking about the paper \cite{B}. }

\end{document}